\documentclass[a4paper,10pt]{article}
\usepackage{setspace}
\usepackage{diagrams}
\usepackage{geometry}
\usepackage{epsfig}
\usepackage{amsfonts}
\usepackage{amsthm}
\usepackage{latexsym}
\usepackage{amsmath}
\usepackage{amscd}
\usepackage{amssymb}
\usepackage{enumerate}
\usepackage{graphicx,oldgerm}

\usepackage{authblk}

\theoremstyle{definition}
\newtheorem{thm}{Theorem}[section]

\newtheorem{lem}[thm]{Lemma}

\newtheorem{prop}[thm]{Proposition}

\newtheorem{note}[thm]{Notation}
\newtheorem{para}[thm]{}

\DeclareMathOperator{\codim}{\mathrm{codim}}
\DeclareMathOperator{\NL}{\mathrm{NL}}
\DeclareMathOperator{\prim}{\mathrm{prim}}

\DeclareMathOperator{\red}{\mathrm{red}}

\DeclareMathOperator{\Ima}{\mathrm{Im}}

\DeclareMathOperator{\p3}{\mathbb{P}^3}
\DeclareMathOperator{\pn}{\mathbb{P}^{2n+1}}

\DeclareMathOperator{\pr}{\mathrm{pr}}

\DeclareMathOperator{\mo}{\mathcal{O}}

\DeclareMathOperator{\NS}{\mathrm{NS}}

\newcommand{\ov}[1]{\overline{#1}}
\newcommand{\mb}[1]{\mathbb{#1}}
\newcommand{\mc}[1]{\mathcal{#1}}
\newcommand{\mr}[1]{\mathrm{#1}}
\title{Noether-Lefschetz locus and a special case of the variational Hodge conjecture: Using elementary techniques}
\author{Ananyo Dan \thanks{The author has been supported by the DFG under Grant 
KL-$2244/2-1$\\ \newline Humboldt Universit\"{a}t zu Berlin, Institut f\"{u}r Mathematik, Unter den Linden $6$, Berlin $10099$.\\ e-mail: dan@mathematik.hu-berlin.de}}
\date{\today}

\begin{document}
\maketitle
\doublespacing

\begin{abstract}
 Fix integers $n \ge 1$ and $d$ such that $nd>2n+2$. The Noether-Lefschetz locus $\NL_{d,n}$ parametrizes smooth projective hypersurfaces in $\pn$ such that $H^{n,n}(X,\mb{C}) \bigcap H^{2n}(X,\mb{Q})\not=\mb{Q}$.
 An irreducible component of the Noether-Lefschetz locus is locally a Hodge locus. One question is to ask under what choice of a Hodge class $\gamma \in H^{n,n}(X,\mb{C}) \bigcap H^{2n}(X,\mb{Q})$ does 
 the variational Hodge conjecture hold true? In this article we use methods coming from commutative algebra and Hodge theory to give an affirmative answer in the case $\gamma$ is the 
 class of a complete intersection subscheme in $X$ of codimension $n$. 
 Another problem studied in this article is: In the case $n=1$ when is an irreducible component of the Noether-Lefschetz locus nonreduced?
 Using the theory of infinitesimal variation of Hodge
 structures of hypersurfaces in $\p3$, we determine all non-reduced components with codimension less than or equal to $3d$ for $d \gg 0$. Here again our primary tool is commutative algebra.
\end{abstract}
\begin{note}
Throughout this article, $X$ will denote a smooth hypersurface in $\pn$.
 Denote by $H^{n,n}(X,\mb{Q})$ the intersection $H^{n,n}(X,\mb{C}) \bigcap H^{2n}(X,\mb{Q})$ and $H_X$ the very ample line bundle on $X$.
\end{note}

\section{Introduction}
It was first stated by M. Noether and later proved by S. Lefschetz that for a general smooth surface $X$ in $\p3$, the rank of the N\'{e}ron-Severi group, denoted $\NS(X)$ is of rank $1$. 
We can then define the \emph{Noether-Lefschetz locus}, denoted $\NL_{d,1}$, to be the space
of smooth degree $d$ surfaces in $\p3$ with Picard rank greater than $1$. Using Lefschetz $(1,1)$-theorem, one can see that $\NL_{d,1}$ is the space 
of smooth degree $d$ surfaces $X$ such that $H^{1,1}(X,\mb{Q})\not=\mb{Q}$. Similarly, we can define \emph{higher Noether-Lefschetz locus} as follows: Let $n>1$ and $d$ another integer such that $nd>2n+2$. Denote by 
$\NL_{d,n}$ the space of smooth degree $d$ hypersurfaces $X$ in $\pn$ such that $H^{n,n}(X,\mb{Q}) \not=\mb{Q}$. The orbit of the action of the monodromy group on a rational class is finite (see \cite{kap}). Consequently,
$\NL_{d,n}$ is an uncountable union of algebraic varieties (see \cite[\S $3.3$]{v5} for more details). 

Let $L$ be an irreducible component of $\NL_{d,n}$. Then $L$ can be locally studied as the Hodge locus corresponding to a Hodge class. In particular, take $X \in L$, general and consider the space of 
all smooth degree $d$ hypersurfaces in $\pn$, denoted $U_{d,n}$. For $X \in L$, general, there exists $\gamma \in H^{n,n}(X,\mb{Q})$ and an open (analytic) simply connected set $U$ in $U_{d,n}$ containing $X$ such that $L \bigcap U$ 
is the Hodge locus corresponding to $\gamma$, 
denoted $\NL_{d,n}(\gamma)$ (see \cite[\S $5.3$]{v4} for more details). Before we state the first main result in this article, we fix some notations. 
Given a Hilbert polynomial P, of a subscheme $Z$, in $\pn$, denote by $H_P$ the corresponding Hilbert scheme.
 Denote by $Q_d$ the Hilbert polynomial of a degree $d$ hypersurface in $\pn$. The flag Hilbert scheme $H_{P,Q_d}$
 parametrizes all pairs $(Z,X)$, where $Z \in H_P, X \mbox{ is a smooth degree } d \mbox{ hypersurface in } \pn$ containing $Z$.
For any $n \ge 1$ we prove the following theorem which is a special case of the variational Hodge conjecture:
\begin{thm}\label{ele2}
Let $Z$ be a complete intersection subscheme in $\pn$ of codimension $n+1$. Assume that there exists a smooth hypersurface in $\pn$, say $X$, containing $Z$, of degree
$d>\deg(Z)$. For the cohomology class $\gamma=a[Z] \in H^{n,n}(X,\mb{Q}), a \in \mb{Q}$,
$\gamma$ remains of type $(n,n)$ if and only if $\gamma$ remains an algebraic cycle. In particular, $\ov{\NL_{d,n}(\gamma)}$ (closure taken in $U_{d,n}$) is isomorphic to an irreducible component of 
$\pr_2H_{P,Q_d}$ which parametrizes all degree $d$ hypersurfaces in $\pn$ containing a complete intersection subcheme with Hilbert polynomial $P$, 
where $P$ (resp. $Q_d$) is the Hilbert polynomial of $Z$ (resp. $X$).
 \end{thm}
In \cite{A1}, Otwinowska proves this statement for $d \gg 0$. Furthermore, in the case $n=1$,  we prove:
\begin{thm}\label{ele1}
 Let $d \ge 5$ and $\gamma$ is a divisor in a smooth degree $d$ surface of the form $\sum_{i=1}^r a_i[C_i]$ with $C_i$ distinct integral curves 
 for all $i=1,...,r$  and $d>\sum_{i=1}^r a_i\deg(C_i)+4$. Then the following are true:
 \begin{enumerate}[(i)]
  \item If $r=1$ and $\deg(C_1)<4$ then $\overline{\NL_{d,1}(\gamma)}$ (closure taken under Zariski topology on $U_{d,1}$) 
  is reduced. In particular, $\ov{\NL_{d,1}(\gamma)}$ is an irreducible component of $\pr_2(H_{P,Q_d})$, the space parametrizing all degree $d$ surfaces containing a 
  reduced curve with the same Hilbert polynomial as $C_1$, which we denote by $P$.
  \item Suppose that $r>1$. For $d \gg 0$, every irreducible component $L$ of $\NL_{d,1}$ of codimension at most $3d$ is locally of the form $\NL_{d,1}(\gamma)$ with $\gamma$
  as above, $\deg(C_i) \le 3$ and $\overline{\NL_{d,1}(\gamma)}_{\red}= \bigcap_{i=1}^r \overline{\NL_{d,1}([C_i])}_{\red}$. 
  Moreover, $\ov{\NL_{d,1}(\gamma)}$ is non-reduced if and only if there exists a pair $(i,j)$, $i \not=j$ such that $C_i.C_j \not=0$.
 \end{enumerate}
\end{thm}

\section{Proof of Theorem \ref{ele2}}

\begin{note}
 Denote by $S_n^k$ the degree $k$-graded piece of $H^0(\mo_{\pn}(k))$. Define $S_n:=\oplus_{k \ge 0} S_n^k$. Let $X$ be a smooth degree $d$ hypersurface in $\pn$, defined by an equation $F$. 
 Denote by $J_F$, the Jacobian ideal of $F$ generated as an $S_n$-module by the partial derivatives of $F$ with respect to $\frac{\partial}{\partial X_i}$ for $i=1,...,2n+1$, where $X_i$ are the coordinates of $\pn$.
 Define, $R_F:=S_n/J_F$. For $k \ge 0$, let $J_F^k$ (resp. $R_F^k$) symbolize the degree $k$-graded piece of $J_F$ (resp. $R_F$).
 \end{note}
 
 \begin{para}\label{ele6}
 We now recall some standard facts about Hodge locus. Let $X$ be a smooth projective hypersurface in $\pn$ of degree $d$.
Recall, there is a natural morphism from 
$H^{n,n}(X)$ to $H^{n,n}(X)_{\prim}$, where $H^{n,n}(X)_{\prim}$ denotes the primitive cohomology on $H^{n,n}(X)$ 
(see \cite[\S $6.2, 6.3$]{v4} for more on this topic). Denote by $\gamma_{\prim}$ the image of $\gamma$ under this morphism. Using the Lefschetz decomposition theorem, one can see that $\NL_{d,n}(\gamma)$ coincides
with $\NL_{d,n}(\gamma_{\prim})$ i.e., $\gamma$ remians of type $(n,n)$ if and only if so does $\gamma_{\prim}$.
\end{para}

 \begin{para}
 Now, $K_{\pn}=\mo_{\pn}(-2n-2)$, $H^0(K_{\pn}(2n+2))=H^0(\mo_{\pn}) \cong \mb{C}$ generated by 
 \[\Omega:=X_0...X_{2n+1}\sum_i (-1)^i \frac{dX_0}{X_0} \wedge ... \wedge \frac{d\hat{X_i}}{X_i} \wedge ... \wedge \frac{dX_{2n+1}}{X_{2n+1}},\]where the $X_i$ are homogeneous coordinates on $\pn$.
 Recall, for the closed immersion $j:X \to \pn$, denote by $H^{2n}(X,\mb{Q})_{\mr{van}}$, the kernel of the Gysin morphism $j_*$ from $H^{2n}(X,\mb{Q})$ to $H^{2n}(\pn,\mb{Q})$. 
 Now, \cite[Theorem $6.5$]{v5} tells us that there is a surjective map,
 \[\alpha_{n+1}:H^0(\pn,\mo_{\pn}((n+1)d-2n-2)) \to F^{n+1}H^{2n+1}(\pn \backslash X,\mb{C}) \cong F^nH^{2n}(X,\mb{C})_{\mr{van}}\]
 which sends a polynomial $P$ to the resisue of the meromorphism form $P\Omega/F^{n+1}$, where $F$ is the defining equation of $X$ (see \cite[\S $6.1$]{v5} for more).
 Finally, \cite[Theorem $6.10$]{v5} implies that $\alpha_{n+1}$ induces an isomorphism between $R_F^{(n+1)d-(2n+2)}$ and $H^{n,n}(X)_{\prim}$.
\end{para}

\begin{para}
 We now recall a theorem due to Macaulay which will be used throughout this article. A sequence of homogeneous polynomials $G_i \in S_n^{d_i}, i=0,...,2n+1$ 
 with $d_i>0$ is said to be \emph{regular} if the $G_i$ have no common zero. Denote by $I_G$ the 
 ideal in $S_n$ generated by the polynomials $P_i$ for $i=0,...,2n+1$. Denote by $H_G$ the quotient $S_n/I_G$ and by $H_G^i$ the degree $i$ graded piece in $H_G$.
 \begin{thm}[Macaulay]\label{ele21}
 Let $N:=\sum_{i=0}^{2n+1} d_i - 2n-2$. Then, the rank of $H_G^N=1$ and for every integer $k$, the pairing, $H_G^k \times H_G^{N-k} \to H_G^N$ is perfect.
  \end{thm}
See \cite[Theorem $6.19$]{v5} for the proof of the statement.
\end{para}

\begin{para}\label{ele3}
Denote by $P \in S_n^{(n+1)d-(2n+2)}$ such that $\alpha_{n+1}(\bar{P})=\gamma$.
 Using \cite[Theorem $6.17$]{v5}, we observe that $T_X\NL_{d,n}(\gamma)$ is isomorphic to the preimage of $\ker(.\bar{P}:R_F^d \to R_F^{(n+2)d-(2n+2)})$ under the natural quotient morphism from $S_n^d \to S_n^d/J_F^d$.
 \end{para}
 
 \begin{para}
 It is easy to see that for any $\gamma' \in H^{n,n}(X,\mb{Q})$,  $\NL_{d,n}(\gamma')=\NL_{d,n}(a'\gamma')$ for any $a' \in \mb{Q}$, non-zero. For the rest of this section, we assume $\gamma=[Z]$, 
 where $Z$ is as in the statement of the theorem.
 \end{para}
 
 \begin{note}\label{ele22}
 Denote by $N:=(n+1)d-(2n+2)$. Since $X$ is smooth, the corresponding Jacobian ideal $J_F$ can be generated by a regular sequence of $2n+2$ polynomials $G_i$ of degree $d-1$. 
 Using Theorem \ref{ele21}, we see that there exists a perfect pairing $R_F^k \times R_F^{2N-k} \to R_F^{2N}$ for all $k \le 2N$ and $R_F^{2N}$ is one dimensional complex vector space. 
  Denote by $T_0'$, the subspace of $R_F^{N}$ which is the kernel under the multiplication map,
 $.P:R_F^{N} \to R_F^{2N}.$ 
 Denote by $T_0$ the preimage of $T_0'$ in $S_n^N$ under the natural projection map from $S_n^N$ to $R_F^N$.
 Define $T_1$ the subspace of $S_n$, a graded $S_n$-module such that for all $t \ge 0$, the $t$-graded piece of $T_1$, denoted $T_{1,t}$ is the largest subvector space of $S_n^t$ such that $T_{1,t}\otimes S_n^{N-t}$
 is contained in $T_0$ for $t<N$, $T_{1,N}=T_0$ and $T_{1,N+t}=T_0 \otimes S_n^t$ for $t>0$. 
 \end{note}
 
 \begin{para}
  It follows from the perfect pairing above that $\dim S_n^N/T_{1,N}=\dim R_F^N/T_0'=1$. Using the definition of $T_1$, it follows, 
  \[S_n^k/T_{1,k} \times S_n^{N-k}/T_{1,N-k} \to S_n^N/T_{1,N}\]is a perfect pairing. Hence, $\dim S_n^d/T_{1,d}=\dim S_n^{N-d}/T_{1,N-d}$.  
 \end{para}

\begin{lem}\label{le1}
    The tangent space $T_X(\NL_{d,n}(\gamma))$ coincides with $T_{1,d}$.
   \end{lem}

   \begin{proof}
    Note that $H \in T_{1,d}$ if and only if $\bar{H} \otimes R_F^{N-d}$ is contained in $T_0'$ 
    which by definition is equivalent to $\bar{P}\bar{H} \otimes R_F^{N-d}=0$ in $R_F^{2N}$. Using the perfect pairing \ref{ele22} we can conclude that 
    $\bar{P}\bar{H}=0$ in $R_F^{N+d}$.
    This is equivalent to $H \in T_X(\NL_{d,n}(\gamma))$.
       \end{proof}

       \begin{para}
    Suppose that $Z$ is defined by $n+1$ polynomials $P_0,...,P_{n}$. Since $Z \subset X$, we can assume that there 
    exist polynomials $Q_0,...,Q_{n}$ of degree $d-\deg P_i$, respectively such that $X$ is defined by a polynomial
    of the form $P_0Q_0+...+P_{n}Q_{n}$. Let $I$ be the ideal in $S_n$ generated by $P_0,...,P_{n}$ and $Q_0,...,Q_{n}$.
               \end{para}

               \begin{prop}\label{le6}
                The $k$-graded pieces, $T_{1,k}=I_k$ for all $k \le N$.                
               \end{prop}

\begin{proof}
 Denote by $Z_1$ the subschemes in $\pn$, defined by $Q_0=P_1=...=P_n=0$. Since $Z \bigcup Z_1$ is the intersection of $X$ and $\{P_1=...=P_n=0\}$, then 
 $[Z]=-[Z_1] \mod \mb{Q}H_X^n$ in the cohomology group $H^{n,n}(X,\mb{Q})$. So, $[Z]_{\prim}=-[Z_1]_{\prim}$. 
 Denote by $Z_2$ the subvariety defined by $Q_0=...=Q_n=0$.
 Proceeding similarly, we get $[Z]_{\prim}=a[Z_2]_{\prim}$ for some integer $a$. Using \cite[$4.a.4$]{GH}, we have $(P_0,...,P_n,Q_0,...,Q_n) \subset T_1$. 
 Since $X$ is smooth the sequence $\{P_0,...,P_n,Q_0,...,Q_n\}$ is a regular sequence. 
 Using Theorem \ref{ele21} we can conclude that $\dim S_n^N/I_N=1$, where $I_N$ denotes the degree $N$ graded piece of $I$
 and \[S/I|_k \times S/I|_{N-k} \to S/I|_N\] is perfect pairing. 
 So, $I$ is Gorenstein of socle degree $N$ contained in $T_1$ which is Gorenstein of the same socle degree.
  So, $T_{1,k}=I_k$ for all $k \le N$.
\end{proof}

\begin{para}
 The parameter space, say $H$ of complete intersection subschemes in $\pn$ of codimension $n+1$, defined by $n+1$ polynomials of degree $\deg(P_i)$, respectively is irreducible. In particular, it is an open subscheme of 
 \[\mb{P}(S_n^{\deg P_0}) \times ... \times \mb{P}(S_n^{\deg P_n})\]
 which is irreducible. Denote by $R_0$ the Hilbert polynomial of $Z$ as a subscheme in $\pn$. Consider the flag Hilbert scheme $H_{R_0,Q_d}$ and the projection map $\pr_1$ 
 which is the projection onto the first component. 
 Since the generic fiber of $\pr_1$
 is isomorphic to $\mb{P}(I_d(Z))$ for the generic subscheme $Z$ on $\pr_1H_{R_0,Q_d}$, it is irreducible, where $I_d(Z)$ is the degree $d$ graded piece of the ideal, $I(Z)$, of $Z$. So, 
 there exists an unique irreducible component in $H_{R_0,Q_d}$ such that the image under $\pr_1$ of this component coincides with $H$. For simplicity of notation, 
 we denote by $H_{R_0,Q_d}$ this irreducible component, since we are interested only in this scheme. 
\end{para}

\begin{para}[Proof of Theorem \ref{ele2}]\label{ele8}
 Using basic deformation theory and Hodge theory, we can conclude that $\pr_2(H_{R_0,Q_d})$ is contained in $\overline{\NL_{d,n}(\gamma)}$.
 So, \[\codim \pr_2(H_{R_0,Q_d}) \ge \codim \NL_{d,n}(\gamma) \ge \codim T_X \NL_{d,n}(\gamma).\]
 Now, there is a natural morphism, denoted $p$ from $T_{1,d}$ to $H_{Q_d}$ which maps $F_1$ to the zero locus of $F_1$.
    Since every element of $T_{1,d}$ defines a hypersurface containing a subscheme with Hilbert polynomial $R_0$, $\overline{\pr_2(H_{R_0,Q_d})}$ 
    contains $\overline{\Ima p}$. Since the zero locus of a polynomial is invariant under multiplication by a scalar,
    \[\dim T_{1,d}=\dim \overline{\Ima p}+1.\]Finally,
    \[\codim \pr_2(H_{R_0,Q_d})=\dim \mb{P}(H^0(\mo_{\pn}(d)))-\dim \overline{\pr_2(H_{R_0,Q_d})} \le\] \[\le (h^0(\mo_{\pn}(d))-1)-\dim \ov{\Ima p} \le h^0(\mo_{\pn}(d))-\dim T_{1,d}=\codim T_X\NL_{d,n}(\gamma)\]
    where the last equality follows from Lemma \ref{le1}. This proves Theorem \ref{ele2}. 
    \end{para}
    
    \begin{para}\label{ele0}
        Furthermore, note that $\ov{\NL_{d,1}(\gamma)}$ is reduced and parametrizes all degree $d$ surfaces in $\p3$ containing a complete intersection curve with the
        Hilbert polynomial $P$. This is a part of Theorem \ref{ele1}(ii).
 \end{para}

\section{Proof of Theorem \ref{ele1}}

\begin{para}\label{ele9}
 If $C_1$ is a complete intersection curve then reducedness of $\NL_{d,1}([C_1])$ follows from \ref{ele0}.
 
 If $C_1$ is an integral curve, $\deg(C_1) < 4$ and $C_1$ not complete intersection then $C_1$ is a twisted cubic. Recall, the twisted cubic $C_1$ is generated by $3$ polynomials of degree $2$ each. 
 Suppose $P$ is a polynomial in $T_1$
 such that $P$ is not contained in $I(C_1)$. Since the zero locus of the ideal, say $I'$ generated by $I(C_1)$ and $P$ is non-empty, $J_F$ (which is base point free) is not contained in $I'$. 
 Since the Jacobian ideal, $J_F \subset T_1$, we can show that there is a regular sequence in $T_1$
 consisting of $4$ elements, two of which are the generators of $I(C_1)$, the third one is $P$ and the forth is an element in $J_F^{d-1}$. Since, $\codim T_{1,2d-4}=1$, Theorem \ref{ele21} implies 
 $2+2+\deg(P)+d-1-4\ge 2d-4$. So, $\deg(P) \ge d-3$.
 By Proposition \ref{le6}, \[\codim T_X(\NL_{d,1}(\gamma)=\codim T_{1,d}=\codim T_{1,d-4}=\]\[=\codim I_{d-4}(C_1)=3.(d-4)+1=3d-11.\]
 \end{para}
 
 \begin{para}
  Using basic deformation theory and Hodge theory we can conclude that there exists an unique irreducible component $H$ of $H_{P,Q_d}$ whose generic
  element is $(C,X)$, where $C$ is a twisted cubic contained in $X$ such that $\pr_2(H)$ is contained in $\ov{\NL_{d,1}(\gamma)}$.
  It is easy to compute that $\codim(\pr_2(H))=3d-11$. 
  So, \[3d-11\ge \codim \ov{\NL_{d,1}(\gamma)} \ge \codim T_X(\NL_{d,1}(\gamma))=3d-11.\] 
  So, $\ov{\NL_{d,1}(\gamma)}$ is reduced and parametrizes smooth degree $d$ surfaces containing a twisted cubic.
  This finishes the proof of $(i)$.
\end{para}

\begin{para}
 We now recall a result due to Otwinowska that will help us make the characterization of the irreducible components of $\NL_{d,1}$ as in Theorem \ref{ele1}(ii).
 \begin{thm}[{\cite[Theorem $1$]{ot}}]\label{gh22}
 Let $\gamma$ be a Hodge class on a smooth degree $d$ surface. There exists $C \in \mathbb{R}_+^*$ depending only on $r$ such that for $d \ge C(r-1)^8$ if 
 $\codim \NL_{d,1}(\gamma) \le (r-1)d$ 
 then $\gamma_{\prim}=\sum_{i=1}^ta_i[C_i]_{\prim}$ where $a_i \in \mathbb{Q}^*$, $C_i$ are integral curves and $\deg(C_i) \le (r-1)$ for
 $i=1,...,t$ for some positive integer $t$.
\end{thm}
This implies the following:
\end{para}

\begin{prop}
 Let $d \gg 0$, $\gamma$ be a Hodge class in a smooth degree $d$ surface in $\p3$ such that $\codim \ov{\NL_{d,1}(\gamma)} \le 3d$. Then there exists 
 integral curves $C_1,...,C_t$ of degree at most $3$ such that $\gamma=\sum_{i=1}^t a_i[C_i]+bH_X$ for some integers $a_i, b$ and $\NL(\gamma)_{\red}$
 is the same as $\bigcap_{i=1}^t \NL([C_i])_{\red}$.
\end{prop}

\begin{proof}
 Let $X \in \NL_{d,1}(\gamma)$. There exists a maximal $\mb{Q}$-vector space $\Lambda \subset H^2(X,\mb{Q})$ such that $\Lambda$ remains of type $(1,1)$ in $\NL_{d,1}(\gamma)$
 i.e., $\NL_{d,1}(\gamma)_{\red}=\bigcap_{\gamma_i \in \Lambda} \NL_{d,1}(\gamma_i)_{\red}$. There exists a surface $X' \in \NL_{d,1}(\gamma)$
 such that the N\'{e}ron-Severi group $\mr{NS}(X')$ is the translate (under deformation from $X$ to $X'$) 
 of $\Lambda$ in $H^2(X',\mb{Z})$ which we again denote by $\Lambda$ for convinience. Then, Theorem \ref{gh22} implies that any $\gamma \in \Lambda$
 is of the form $\sum_i a_i[C_i]+bH_X$ with $\deg(C_i) \le 3$. So, $\Lambda$ is generated by classes of curves of degree at most $3$ and $H_X$.
 Note that the classes of these curves are also contained in $\Lambda$ since $\Lambda$ is the complete N\'{e}ron-Severi group of $X'$.
 This proves the proposition, which is also the first part of Theorem \ref{ele1}(ii). 
\end{proof}

\begin{para}\label{a5e}
We now come to the proof of the final part of the theorem.
Suppose now that $\gamma$ is as in the above proposition i.e., of the form $\sum_{i=1}^t a_i[C_i]+bH_X$ such that $\NL(\gamma)_{\red}=\bigcap_{i=1}^t \NL([C_i])_{\red}$.
 Denote by $\bar{P}_i$ the element in $R_F^{2d-4}$ such that $\alpha_2(\bar{P}_i)=[C_i]_{\prim}$ for $i=1,...,t$. Since $\alpha_2$ is a linear map, 
 $\alpha_2(\sum_{i=1}^r a_i\bar{P}_i)=\sum_i a_i[C_i]_{\prim}$. Denote by
 $\bar{P}:=\sum_{i=1}^r a_i\bar{P}_i$. So, $\alpha_2(\bar{P})=\gamma$. Denote
 by $T_{1,d-4}^{[C_i]}$ the corresponding $T_{1,d-4}$ in \ref{ele22} obtained by replacing $P$ by $\bar{P}_i$ for $i=1,...,r$.
 Note that $\codim T_X\NL_{d,1}(\gamma)=\codim T_{1,d}=\codim T_{1,d-4},$ where the last equality is due to perfect pairing. Note that, $\bigcap_{i=1}^r T_{1,d-4}^{[C_i]} \subset T_{1,d-4}$ because $\bar{P}=\sum_i a_i\bar{P}_i$,
 so  $\bigcap_{i=1}^r \ker \bar{P}_i \subset \ker \bar{P}$. Therefore, $\codim T_X\NL_{d,1}(\gamma) \le \codim I_{d-4}(\bigcup_{i=1}^r C_i).$
 \end{para}

 Before we go to the last step of the proof we need the following computation:
 
 \begin{lem}\label{a4e}
Let $d \ge 5$ and $C$ be an effective divisor on a smooth degree $d$ surface $X$ of the form $\sum_i a_iC_i,$ where $C_i$ are integral curves with $\deg(C)+4 \le d$.
Then, $\dim |C|=0,$ where $|C|$ is the linear system associated to $C$.
\end{lem}

\begin{proof}
Let $C=\sum_ia_iC_i$ with $C_i$ integral. Then, \[\deg((\mathcal{O}_X(C)|_C \otimes \mo_C)|_{C_i})=a_iC_i^2+\sum_{j \not= i} a_jC_i.C_j.\]
Denote by $e_i:=\deg(C_i)$. 
Using the adjunction formula and the fact that $K_X \cong \mathcal{O}_X(d-4)$, we have that 
\begin{eqnarray*}
\deg((\mathcal{O}_X(C)|_C \otimes \mo_C)|_{C_i})&=&2a_i\rho_a(C_i)-2a_i-(d-4)a_ie_i+\sum_{j \not= i}a_jC_i.C_j\\
&\le&a_i(e_i^2-(d-1)e_i)+\sum_{j \not=i}a_jC_iC_j\\
&\le&a_i(e_i^2-3e_i-e_i\sum_ja_je_j)+\sum_{j \not=i}a_je_ie_j.
\end{eqnarray*}
The first inequality follows from the bound on the genus of a curve in $\mathbb{P}^3$ in terms of its degree (see \cite[Example $6.4.2$]{R1}).
The second inequality follows from the facts that $d \ge \deg(C)+2$ and $C_i.C_j \le e_ie_j$. It then follows directly that $\deg((\mathcal{O}_X(C)|_C \otimes \mo_C)|_{C_i})<0$.
This implies that $h^0(C_i,(\mathcal{O}_X(C)|_C \otimes \mc{O}_C)|_{C_i})=0$ for all $i$. This implies that 
$h^0(C,\mathcal{O}_X(C)|_C \otimes \mo_C)=0$. 
Since $h^1(\mathcal{O}_X)=0$ (by Lefschetz hyperplane section Theorem) and $h^0(\mathcal{O}_X)=1$, using 
the long exact sequence associated to the short exact sequence 
\begin{equation}\label{eq2}
0 \to \mathcal{O}_X \to \mathcal{O}_X(C) \to \mathcal{O}_X(C)|_C \otimes \mo_C \to 0
\end{equation}
we get that $h^0(\mathcal{O}_X(C))=1$. Since $|C|=\mb{P}(H^0(\mo_X(C)))$, the lemma follows.
\end{proof}

 \begin{para}[Proof of Theorem \ref{ele1}]
  Using Proposition \ref{le6} and \ref{ele9}, $\bigcap_{i=1}^r T_{1,d-4}^{[C_i]}=I_{d-4}(\bigcup_{i=1}^r C_i)$ is contained in $T_X\NL_{d,1}(\gamma)$. 
  Denote by $P_i$ the Hilbert polynomial of $C_i$ for $i=1,...,t$. By Theorem \ref{ele1}(i), there exists an irreducible component of $H_{P_i,Q_d}$
  such that its image under the natural projection morphism $\pr_2$ (onto the second component) is isomorphic to $\ov{\NL_{d,1}([C_i])}_{\red}$.
  So, there exists an irreducible component, say $H_\gamma$ of $H_{P_1,Q_d} \times_{H_{Q_d}} ... \times_{H_{Q_d}} H_{P_t,Q_d}$ such that 
  $\pr_2(H_{\gamma})_{\red}=\ov{\NL(\gamma)}_{\red}$, where $\pr_2$ is the natural morphism from $H_\gamma$ to $H_{Q_d}$. Denote by $L_\gamma:=\pr(H_\gamma)$,
  where $\pr$ is the natural projection morphism to $H_{P_1} \times ... \times H_{P_t}$. 
  A generic $t$-tuple of curves $(C_1,...,C_t) \in H_{P_1} \times ... \times H_{P_t}$ does not intersect each other.
  Since there exists $i, j, i \not= j$ such that $C_i \bigcap C_j \not= \emptyset$, we have $\dim L_\gamma < \sum_{i=1}^t \dim H_{P_i}$. 
  Lemma \ref{a4e} implies that $\dim |C_i|=0$ for $i=1,...,t$.
  It is then easy to see that 
  $\codim \NL_{d,1}(\gamma)=\codim I_d(\bigcup_{i=1}^t C_i)-\dim L_\gamma$.
  If $\codim I_{d-4}(\bigcup_{i=1}^t C_i) \stackrel{\dagger}{\le} \codim I_d(\bigcup_{i=1}^t C_i)-\sum_{i=1}^t \dim H_{P_i}$ then \[\codim T_X\NL_{d,1}(\gamma) \le \codim 
  I_{d-4}(\bigcup_{i=1}^t C_i) \stackrel{\dagger}{\le} \codim I_d(\bigcup_{i=1}^t C_i)-\sum_{i=1}^t \dim H_{P_i}\]
  \[<\codim I_d(\bigcup_{i=1}^t C_i)-\dim L_\gamma=\codim \NL_d(\gamma),\]where the first inequality follows from \ref{a5e}.
  Since $d \gg 0$, using the Hilbert polynomial of $\bigcup C_i$, the inequality $\dagger$ is equivalent to $\sum_{i=1}^t \dim H_{P_i} \le 4\sum_{i=1}^t \deg(C_i)$. 
  Since $\deg(C_i)<4$ and $C_i$ is integral, it is easy to compute that $\dim H_{P_i}$ is infact equal to $4\deg(C_i)$.
  This proves (ii). Hence, completes the proof of the theorem.
 \end{para}

\bibliographystyle{alpha}
 \bibliography{researchbib}
 
\end{document}